\documentclass[letterpaper, 10 pt, conference]{ieeeconf}  

\IEEEoverridecommandlockouts                              
\usepackage{graphics} 
\usepackage{epsfig} 
\usepackage{mathptmx} 
\usepackage{times} 
\usepackage{amsmath} 
\usepackage{amssymb}  
\usepackage{subfigure}
\usepackage{color}
\usepackage{flushend}
\usepackage{eurosym}
\usepackage{amsmath,theorem}
\usepackage{amssymb}
\usepackage{amsfonts}
\usepackage{graphics}
\usepackage{psfrag}
\usepackage{array}
\usepackage{shortvrb}
\usepackage{epsf}
\usepackage{graphicx}
\usepackage{rotating}
\usepackage{cite}
\usepackage{wasysym}
\usepackage[usenames]{xcolor}

\usepackage{graphicx}
\usepackage{eurosym}
\usepackage{amssymb}
\usepackage{amsmath,mathtools}
\usepackage{amsfonts}
\usepackage{epstopdf}
\usepackage{epsf,subfigure}
\usepackage{psfrag}
\usepackage{graphics}
\usepackage{color} 

\newtheorem{definition}{Definition}
\newtheorem{theorem}{Theorem}

\newtheorem{lemma}{Lemma}

\newcommand{\qed}{\nobreak \ifvmode \relax \else
      \ifdim\lastskip<1.5em \hskip-\lastskip
      \hskip1.5em plus0em minus0.5em \fi \nobreak
      \vrule height0.75em width0.5em depth0.25em\fi}

\title{\LARGE \bf
Computation of Linear Comparison Equations for Stability Analysis of Interconnected Systems
}

\author{Soumya Kundu and Marian Anghel
\thanks{*This work was supported by the U.S. Department of Energy
through the LANL/LDRD Program.}
\thanks{$^{1}$Soumya Kundu is with the Center for Nonlinear Studies and Information Sciences Group (CCS-3), Los Alamos National Laboratory, Los Alamos, USA
        {\tt\small soumya@lanl.gov}}%
\thanks{$^{2}$Marian Anghel is with the Information Sciences Group (CCS-3), Los Alamos National Laboratory, Los Alamos, USA
        {\tt\small manghel@lanl.gov}}%
}

\begin{document}

\maketitle
\thispagestyle{empty}
\pagestyle{empty}

\begin{abstract}

Sum-of-squares (SOS) methods have been shown to be very useful in computing polynomial Lyapunov functions for systems of reasonably small size. However for large scale systems it is necessary to use a scalable alternative using vector Lyapunov functions. Earlier works have shown that under certain conditions the stability of an interconnected system can be studied through suitable comparison equations. However finding such comparison equations can be non-trivial. In this work we propose an SOS based systematic procedure to directly compute the comparison equations for interconnected system with polynomial dynamics. With an example of interacting Van der Pol systems, we illustrate how this facilitates a scalable and parallel approach to stability analysis.

\end{abstract}

\section{INTRODUCTION}

Lyapunov functions methods have long been used in studying stability properties of dynamical systems \cite{Lyapunov:1892,Haddad:2008}. Finding a Lyapunov function for a given dynamical system, however, is often not an easy task. Recent advances in sum-of-squares (SOS) methods and semi-definite programming, \cite{sostools13,Antonis:2005a,Sturm:1999}, have enabled algorithmic construction of polynomial Lyapunov functions \cite{Wloszek:2003,Anghel:2013}. However such sum-of-squares based computational methods become intractable as the system size grows to larger than 6-8 states \cite{Antonis:2012,Anderson:2010}. 

It is useful to model large-scale systems in the form of many interacting subsystems and study the stability of the full interconnected system using only the subsystem Lyapunov functions. There are different functional forms for the Lyapunov function of the interconnected system, such as a scalar Lyapunov function expressed as a weighted sum of the subsystem Lyapunov functions, or applications of vector Lyapunov functions and comparison principles \cite{Siljak:1972,Weissenberger:1973, Michel:1983,Araki:1978}. Particularly the formulations using vector Lyapunov functions are computationally very attractive because of their parallel structure and scalability. Based upon the results on comparison equations \cite{Conti:1956,Brauer:1961,Beckenbach:1961}, the authors in \cite{Bellman:1962,Bailey:1966} introduced the concept of vector Lyapunov functions. It was shown that if the subsystem Lyapunov functions and the interactions satisfy certain conditions, then the stability of the interconnected system can be studied by analyzing the stability of a set of linear ordinary differential equations. However computing these comparison equations, for a given interconnected system, still remained a challenge. In absence of suitable computational tools, analytical insights were used to build those comparison equations, such as the trigonometric inequalities in power systems network \cite{Jocic:1978}.

In this work we use the sum-of-squares and semi-definite programming methods to study the stability of an interconnected system by computing the comparison equations. While this approach is applicable to any generic dynamical system, we choose a randomly generated network of modified\footnote{We choose the Van der Pol `oscillator' parameters in such a way that these have a stable equilibrium at origin.} Van der Pol oscillators for illustration. Each Van der Pol oscillator can be represented as a two-state system with state dynamic equations as polynomials of degree three \cite{van1926}. The network is then decomposed into many interacting subsystems. Each subsystem parameters are so chosen that individually each subsystem is stable, when the disturbances from neighbors are zero. SOS based expanding interior algorithm \cite{Wloszek:2003,Anghel:2013} is used to obtain estimate of region of attraction as sub-level sets of polynomial Lyapunov functions for each such subsystem. Finally SOS optimization is used to compute the linear comparison equation to certify stability of the network under disturbances. Following some brief background in Sec.\,\ref{S:background} we outline the problem statement in Sec.\,\ref{S:problem}. We present the SOS-based direct approach to computing the comparison equations in Sec.\,\ref{S:stability}. Sec.\,\ref{S:example} shows an application of comparison equations to stability analysis of a network of Van der Pol systems. We conclude the article in Sec.\,\ref{S:conclusion}.

\section{BASIC CONCEPTS AND BACKGROUND}
\label{S:background}

\subsection{Lyapunov Stability Methods}
\label{S:Lyap}
Let us consider the dynamical system 
\begin{align}\label{E:f}
&\dot{x}\left(t\right) = f\left(x\left(t\right)\right),\quad t\geq 0,~x\in\mathbb{R}^n, ~f\left(0\right)=0\,,
\end{align}
with an equilibrium at the origin\footnote{Note that this is not a restrictive assumption, since by shifting of state variables, the origin can always be made an equilibrium point.}, and $f:\mathbb{R}^n\rightarrow \mathbb{R}^n$ is locally Lipschitz. 
The important notions of stability are:
\begin{definition}\label{D:stability}
The equilibrium point at origin is called 
\begin{enumerate}
\item stable in the sense of Lyapunov (i.s.L) if
\begin{align*}
\forall \epsilon\!>\!0, \exists \delta\!>\!0 ~\text{s.t.}~ \left\|x(0)\!\right\|_2\!<\!\delta\!\!\implies\!\!\left\|x(t)\!\right\|_2\!<\!\epsilon~\forall t,
\end{align*}
\item asymptotically stable if it is stable i.s.L, and
\begin{align*}
\exists \tilde{\delta}\!>\!0 ~\text{s.t.}~ \left\|x(0)\!\right\|_2\!<\!\tilde{\delta}\!\implies\!\lim_{t\rightarrow +\infty}\left\|x(t)\!\right\|_2=0,
\end{align*}
\item exponentially stable if it is asymptotically stable, and 
\begin{align*}
\exists b,c,\hat{\delta} \!\!>\! 0 ~\text{s.t.}~
\left\|\!x(0)\!\!\right\|_2\ \!\!\!\!<\!\! \hat{\delta} \!\!\!\implies\!\!\!\! \left\|\!x(t)\!\!\right\|_2\!<\!ce^{-bt}\!\left\|\!x(0)\!\!\right\|_2~\forall t
\end{align*}
\end{enumerate}
\end{definition}

The Lyapunov stability theorem \cite{Lyapunov:1892,Slotine:1991}, also called Lyapunov's first or direct method, presents a sufficient condition of stability through the construction of a certain positive definite function.
\begin{theorem}\label{T:Lyap}
The equilbrium point $x=0$ of the dynamical system in (\ref{E:f}) is stable i.s.L in $\mathcal{D}\subseteq\mathbb{R}^n$, if there exists a continuously differentiable positive definite function {$\tilde{V}:\mathcal{D}\rightarrow \mathbb{R}$} (henceforth referred to as Lyapunov function) such that,
\begin{subequations}\label{E:Lyap}
\begin{align}
\tilde{V}\left(0\right) &= 0\,, \\
\tilde{V}\left(x\right) &>0,\forall x\in\mathcal{D}\backslash{\left\lbrace0\right\rbrace}\,,\\
\text{and, }-\dot{\tilde{V}}\left(x\right)&\geq 0,\forall x\in\mathcal{D}\,.
\end{align}
\end{subequations}
If $\tilde{V}$ satisfies $-\dot{\tilde{V}}(x)>0,\forall x\in\mathcal{D}\backslash{\left\lbrace0\right\rbrace},$ then the equilibrium point at origin is asymptotically stable in $\mathcal{D}$.
Further, the origin is exponentially stable\footnote{We will be referring to $\alpha>0$ in \eqref{E:Lyap_alpha} as the `self-decay rate'.} in $\mathcal{D}\in\mathbb{R}^n$ if 
\begin{align}\label{E:Lyap_alpha}
\exists \alpha>0,~\text{s.t.}~-\dot{\tilde{V}}\left(x\right)&\geq \alpha \tilde{V}\left(x\right),\forall x\in\mathcal{D}\,.
\end{align}
\end{theorem}
Here $\dot{\tilde{V}}(x)=\nabla{\tilde{V}}^Tf(x)$. When there exists such a function $\tilde{V}\left(x\right)$, the region of attraction (ROA) of the stable equilibrium point at origin can be (conservatively) estimated as
\begin{subequations}\label{E:ROA}
\begin{align}
\mathcal{R}&:=\left\lbrace x\in\mathcal{D}\left| \tilde{V}(x)\leq \gamma^{max}\right.\right\rbrace,\\
\text{where,}~\gamma^{max}&:=\arg\max_\gamma\left\lbrace x\in\mathbb{R}^n\left| \tilde{V}(x)\leq\gamma\right.\right\rbrace \subseteq \mathcal{D}.
\end{align}
\end{subequations}
The Lyapunov function can be scaled by $\gamma^{max}$, so that,
\begin{subequations}\label{E:ROA2}
\begin{align}
\mathcal{R}:=&\left\lbrace x\in\mathbb{R}^n\left| {V}(x)\leq 1\right.\right\rbrace,\\
\text{where,}~{V}(x)=& ~{\tilde{V}(x)}/{\gamma^{max}}.
\end{align}
\end{subequations}
Henceforth, for simplicity, we would assume, without any serious loss of generality, that the ROA is estimated to be sub-level set of ${V}(x)=1$.

\subsection{Sum-of-Squares and Positivstellensatz Theorem}
\label{subsec:SOSmethod}
Relatively recent studies have shown that sum-of-squares based optimization techniques can be utilized in finding Lyapunov functions by restricting the search space to sum-of-square polynomials \cite{Wloszek:2003,Parrilo:2000,Tan:2006,Anghel:2013}. Let us denote by $\mathbb{R}\left[x\right]$ the set of all polynomials in $x\in\mathbb{R}^n$. Then,
{\begin{definition}
A multivariate polynomial {$p \in \mathbb{R}[x]$, $x\in\mathbb{R}^n$}, is called a sum-of-squares (SOS) if there exists $h_i\in\mathbb{R}[x]$, $i\in\left\lbrace 1,2,\dots,r\right\rbrace$, such that $p(x) = \sum_{i=1}^r h_i^2(x)$. Further, we denote the set of all SOS polynomials in $x\in\mathbb{R}^n$ by $\Sigma[x]$.
\end{definition}
Checking if $p\in\mathbb{R}[x]$ is an SOS is a semi-definite problem} which can be solved with a MATLAB$^\text{\textregistered}$ toolbox SOSTOOLS \cite{sostools13,Antonis:2005a} along with a semidefinite programming solver such as SeDuMi \cite{Sturm:1999}.
{SOS technique can be used to search for polynomial Lyapunov functions, by translating \eqref{E:Lyap} to equivalent SOS conditions \cite{sostools13,Antonis:2002,Wloszek:2003,Wloszek:2005,Antonis:2005,Antonis:2005a,Antonis:2005b }.}
An important result from algebraic geometry called Putinar's Positivstellensatz theorem \cite{Putinar:1993,Lasserre:2009} helps in translating the SOS conditions into SOS feasibility problems. 
Then the Putinar's Positivestellensatz theorem\footnote{For other versions of the Positivstellensatz theorem please refer to \cite{Lasserre:2009}.} states,
{\begin{theorem}\label{T:Putinar}
Let $\mathcal{K}= \left\lbrace x\in\mathbb{R}^n\left\vert u_1(x) \geq 0, \dots , u_m(x)\geq 0\right.\right\rbrace$ be a
compact set, where $u_j\in\mathbb{R}[x]$, $\forall j\in\left\lbrace 1,\dots,m\right\rbrace$. Suppose
\begin{align}\label{E:Putinar}
\exists\, u\!\in\!\mathbb{R}[x],\,\text{so that},& \left\lbrace \!\!\begin{array}{l}\!\! u\!\in\!\!\left\lbrace \sigma_0 \!\!+\!\! \sum_j\sigma_ju_j\!\!\left\vert\, \sigma_0,\sigma_j\!\!\in\!\Sigma[x],\forall j\!\!\right.\right\rbrace \\
\!\!\&\left\lbrace x\in\mathbb{R}^n\left\vert u(x)\geq 0\!\!\right.\right\rbrace\,\text{is compact.}\end{array}\!\!\right.
\end{align}
If $p(x)\!>\!0,\,\forall x\!\in\!\!\mathcal{K}$, then $p\in \left\lbrace \sigma_0 \!\!+\!\! \sum_j\sigma_ju_j\!\!\left\vert\, \sigma_0,\sigma_j\!\!\in\!\Sigma[x],\forall j\!\!\right.\right\rbrace$.
\end{theorem}
Often for the $u_i,\,\forall i,$ used in this work, the existence of $u(x)$ in \eqref{E:Putinar} would be guaranteed \cite{Lasserre:2009}}.

\subsection{Linear Comparison Principle}
\label{subsec:comparison}
Before finishing this section, let us take a look at a nice result on the ordinary differential equations which helps form the framework of stability analysis of inter-connected systems via vector Lyapunov functions. Noting that all the elements of the vector $e^{At},~ t\geq 0$, where $A=\left[a_{ij}\right]\in\mathbb{R}^{m\times m}$, are non-negative if and only if $a_{ij}\geq 0, i\neq j$, the authors in \cite{Beckenbach:1961,Bellman:1962} proposed the following result:
\begin{lemma}\label{L:comparison}
Let $A=[a_{ij}]\in\mathbb{R}^{m\times m}$ have only non-negative off-diagonal elements, i.e. $a_{ij}\geq 0,~i\neq j$. Then 
\begin{align}\label{E:comp_ineq}
\dot{v}(t)\leq A\,v(t), ~t\geq 0, ~v\in\mathbb{R}^n, ~v(0) = v_0, 
\end{align}
implies $v(t)\leq w(t),~\forall t\geq 0$, where 
\begin{align}\label{E:comp_eq}
\dot{w}(t)= A\,w(t), ~t\geq 0, ~w\in\mathbb{R}^n, ~w(0) = v(0) = v_0. 
\end{align}
\end{lemma}
This result will henceforth be referred to as the `linear comparison principle' and the differential equation in \eqref{E:comp_eq} as the `comparison equation'.

\section{PROBLEM DESCRIPTION}\label{S:problem}
For the rest of this work, let us make the simplifying assumption that the dynamical system in (\ref{E:f}) is in polynomial form\footnote{Non-polynomial dynamics can be recasted into an equivalent polynomial form, with introduction of additional state variables and suitable equality constraints \cite{Antonis:2002, Antonis:2005,Anghel:2013,Anghel:2013b}.}, denoted by $f\in\mathbb{R}[x]^n$, and that the system in \eqref{E:f} is (locally) asymptotically stable.

\subsection{Decomposed System Model}
\label{subsec:decomp}
The dynamical system in \eqref{E:f} can be expressed in the form of $m$ ($\geq 2$) interacting, and asymptotically stable subsystems
\begin{subequations}\label{E:fi}
\begin{align}
\forall i =1,2,&\dots,m,\nonumber\\
\mathcal{S}_i:~&\dot{x}_i = f_i(x_i) + g_i(x), ~ x_i\in\mathbb{R}^{n_i}, ~x\in\mathbb{R}^n\\
&f_i({0})={0},\\
&g_i(\hat{x}_{i})={0},~\forall \hat{x}_{i}\in\left\lbrace x\in\mathbb{R}^n\!\left\vert ~x_j \!=\! 0,\forall j\!\neq\! i\right.\right\rbrace\\
\text{where,}~ x &= \left( x_1^T ,x_2^T, \dots, x_m^T\right)^T \in \mathbb{R}^{n}\\
\text{and}~ n &=\sum_{i=1}^m n_i\, , ~x_i\cap x_j = \emptyset\, . \label{E:non-overlap}
\end{align}
\end{subequations}
Here $x_i$ represents the states that belong to the $i$-th subsystem $\mathcal{S}_i$, $f_i\in\mathbb{R}[x_i]^{n_i}$ denotes the isolated subsystem dynamics, and $g_i\in\mathbb{R}[x]^{n_i}$ represents the neighbor interactions. 

{Let us assume that the interactions can be expressed as
\begin{align}\label{E:gij}
\forall i\in\left\lbrace 1,2,\dots,m\right\rbrace, \quad g_i(x) = \sum_{j\neq i} g_{ij}(x_i, x_j)\,,
\end{align}
where $g_{ij}\in\mathbb{R}[x_i,x_j]^{n_i}$ quantifies how subsystem $\mathcal{S}_j$ affects the dynamics of subsystem $\mathcal{S}_i$. Note that \eqref{E:gij} is not a very restrictive assumption, since given the choice of states $x_i$ of the subsystem $\mathcal{S}_i$, the rest of the subsystems can always be chosen in a way such that \eqref{E:gij} holds.
 We denote by
\begin{subequations}\label{E:Ni}
\begin{align}
&\mathcal{N}_i := \left\lbrace i\right\rbrace\cup\left\lbrace j\left\vert ~\exists \, \left\lbrace x_i,x_j\right\rbrace, ~\text{s.t.}~g_{ij}\left(x_i,x_j\right)\neq 0 \right.\right\rbrace \\
&\text{and }\bar{x_{i}} :=\bigcup_{j\in\mathcal{N}_i}\,x_j\,,
\end{align}\end{subequations}
the set of indices of the subsystems in the neighborhood of $\mathcal{S}_i$ (including the subsystem itself) and the states that belong to this neighborhood, respectively.}

The next step is to characterize the stability properties of the isolated subsystems
\begin{align*}
\forall  i \in\left\lbrace 1,2,\dots,m\right\rbrace,\quad &\dot{x}_i = f_i(x_i),~x_i\in\mathbb{R}^{n_i}  \, .
\end{align*}
by computing a polynomial Lyapunov function $V_i \in \mathbb{R}\left[x_i\right] $ for each $i$, and the corresponding estimate of the ROA as in \eqref{E:ROA2}.
An SOS based \textit{expanding interior algorithm}, \cite{Wloszek:2003,Anghel:2013}, is used to iteratively enlarge the estimate of the ROA by finding a `better' Lyapunov function at each step of the algorithm. At the completion of this iterative algorithm, the stability of each isolated subsystem (assuming no interaction) is quantified by its Lyapunov function $V_i\in\mathbb{R}[x_i]$, with a final estimate of the domain of attraction given by 
 \begin{align}
\mathcal{R}_i^0:= \left\lbrace x_i\in\mathbb{R}^{n_i}\left| V_i(x_i)\leq 1\right.\right\rbrace,~\forall i=1,2,\dots,m\,.
 \end{align}

\subsection{Stability under Interactions}
\label{subsec:inter_stab}
Let us define the domain
\begin{align}\label{E:ROA_isol}
\mathcal{R}^{0} &:=\left\lbrace x\in\mathbb{R}^{n}\left| ~x_i\in\mathcal{R}_i^0,\,~ \forall i=1,2,\dots,m\right.\right\rbrace\, ,
\end{align}
which could be interpreted as the ROA of the `free' interconnected system \eqref{E:fi}, in absence of the all the interactions. The disturbances coming from the neighbors can be expressed by the subsystem Lyapunov function level-sets. While the equilibrium at origin corresponds to the level sets $V_i(0)=0,\forall i$, any disturbance (or initial condition) away from this equilibrium would result in positive level-sets $V_i(x_i(0))=\gamma_i^0\in\left(0,1\right]$ for some or all of the subsystems. 

A necessary and sufficient condition of asymptotic stability (Definition~\ref{D:stability}) can then be translated into the condition
\begin{align}\label{E:cond_asymp}
\forall i, ~V_i(x_i(0))=\gamma_i^0\implies\forall i, ~\lim_{t\rightarrow +\infty}{V}_i(x_i(t))=0\, ,
\end{align}
where $x_i(t),~t>0$, are solutions of the coupled dynamics in \eqref{E:fi}. Even though \eqref{E:cond_asymp} reduces the dimensionality of the problem, it still remains a generally non-trivial problem.
An attractive, and scalable, alternative approach is to construct a vector Lyapunov function $V:\mathbb{R}^n\rightarrow\mathbb{R}^m$
\begin{align}\label{E:vecLyap}
V(x) &:= \left[V_1(x_1)  ~~ V_2(x_2) ~~\dots ~~V_m(x_m)\right]^T , 
\end{align} 
and use a comparison equation to certify if the condition \eqref{E:cond_asymp} holds. Restricting our focus to the linear comparison principle (Lemma~\ref{L:comparison}), the aim is to seek an $A=[a_{ij}]\in\mathbb{R}^{m\times m}$ and a domain $\mathcal{D}\subset\mathcal{R}^0$, such that
\begin{subequations}\label{E:comparison}
\begin{align}
\dot{V}(x)&\leq~ AV(x),~\forall x\in\mathcal{D}\subset\mathcal{R}^0, \label{E:comparison_VAV}\\
\text{where,}\quad & a_{ij}\geq 0~\forall i\neq j\, ,\\
		 &  \text{$A=[a_{ij}]$ is Hurwitz, and} \\
		\quad & \text{$\mathcal{D}$ is invariant under the dynamics \eqref{E:f}.}
\end{align}
\end{subequations}
If there exist a `comparison matrix' $A=[a_{ij}]$ and $\mathcal{D}\subset\mathcal{R}^0$ satisfying \eqref{E:comparison}, then any $x(0)\in\mathcal{D}$ would guarantee exponentially convergence of $V(x(t))$ to the origin (Lemma~\ref{L:comparison}),
\begin{align}
\exists \,b,c \!>\! 0 ~\text{s.t.} ~\left\| V(x(t))\right\|_2 \!< c e^{-bt} \left\| V(x(0))\right\|_2\, , ~\forall t \!>\!0\, ,
\end{align} 
which also translates into exponential convergence of the states themselves \cite{Siljak:1972}. Note that, $\mathcal{D}\subset\mathcal{R}^0$, if exists, presents an estimate of the ROA of the full interconnected system.

\section{COMPUTING THE COMPARISON EQUATION}\label{S:stability}

\subsection{Traditional Approach}\label{subsec:stab_old}
In \cite{Siljak:1972,Weissenberger:1973,Araki:1978,Jocic:1978}, and related works, authors laid out a formulation of the linear comparison equation using certain conditions on the Lyapunov functions and the neighbor interactions. It was observed that if there exists a set of Lyapunov functions, $v_i:\mathbb{R}^{n_i}\rightarrow\mathbb{R}\,,~\forall\,i=1,2,\dots,m,$ satisfying the following conditions
\begin{subequations}\label{E:cond_VLF}
\begin{align}
\forall i\in\left\lbrace 1,2,\dots,m\right\rbrace, ~\exists\,& \tilde{\eta}_{i1},\tilde{\eta}_{i2},\tilde{\eta}_{i3}>0\,~\text{such that,}\nonumber\\
\forall x_i\in\mathcal{D}_i \!\!\subset\!\mathcal{R}_i^0,~&\tilde{\eta}_{i1}\left\|x_i\right\|_2 \leq v_i(x_i) \leq \tilde{\eta}_{i2}\left\|x_i\right\|_2 \label{E:cond_VLF_1}\\
\text{and}~&\left(\nabla{v}_i\right)^T\!\!f_i\, \leq -\tilde{\eta}_{i3}\left\|x_i\right\|_2 \label{E:cond_VLF_2}
\end{align}
\end{subequations}
and if the interaction terms in \eqref{E:gij} satisfy
\begin{align}\label{E:cond_inter}
\forall i\in&\left\lbrace 1,\dots,m\right\rbrace, \,\forall j\!\in\!\mathcal{N}_i\backslash\left\lbrace i\right\rbrace, ~\exists \tilde{\zeta}_{ij}>0\,~\text{such that,}\nonumber \\
&\forall x_i\!\in\!\mathcal{D}_i, \,\forall x_j\!\in\!\mathcal{D}_j, ~\left\|\left(\nabla{v}_i\right)^T\!\!g_{ij}\right\|_2 \leq \tilde{\zeta}_{ij}\left\|x_j\right\|_2, 
\end{align}
then the following comparison equation can be formed,
\begin{subequations}\label{E:cond_A}
\begin{align}
\forall\, x(t)\!\!\in&\mathcal{D},~\dot{v}\left(x\left(t\right)\right) \leq \tilde{A}\,v\left(x\left(t\right)\right), ~ \tilde{A} = \left[\tilde{a}_{ij}\right]\in \mathbb{R}^{m\times m} \label{E:VAV}\\
\text{where,}~  &v(x) = \left[v_1(x_1)~ v_2(x_2)~\dots~v_m(x_m)\right]^T \!\!,\label{E:def_v}\\
&\mathcal{D}=\left\lbrace x\in\mathbb{R}^n\!\!\left|\, x_i\in\mathcal{D}_i,\,\forall i\! \in\!\!\!\left\lbrace 1,\dots,m\right\rbrace\right.\!\!\right\rbrace \!\!\subset\mathcal{R}^0, \\
 \text{and}\quad & \tilde{a}_{ij} = \left\lbrace \!\!\!\begin{array}{cl} -\tilde{\eta}_{i3}/\tilde{\eta}_{i2}, & j\!=\!i\\ \tilde{\zeta}_{ij}/\tilde{\eta}_{j1}, & j\!\in\!\mathcal{N}_i\backslash\left\lbrace i\right\rbrace\\ 0, & j\notin\mathcal{N}_i\end{array}\!\!\!\right.\!\!\!,\quad \forall \,i, \forall j \label{E:aij}
\end{align}
\end{subequations}
If the `comparison matrix' $\tilde{A}=[\tilde{a}_{ij}]$ is Hurwitz, then any invariant domain $\mathcal{R}\subseteq\mathcal{D}$ provides an estimate of a region of exponential stability of the full system \cite{Weissenberger:1973,Jocic:1978}.

\subsection{Motivation for Direct Approach}\label{subsec:stab_new}
While this approach provides very useful analytical insights into the construction of the comparison matrix $\tilde{A}=[\tilde{a}_{ij}]$, it has certain computational issues. This requires finding the bounds in \eqref{E:cond_VLF} and \eqref{E:cond_inter}, and also the Lyapunov functions $v_i,\,\forall\,i,$ that satisfy those. Clearly the polynomial Lyapunov functions, $V_i\,\forall\,i,$ cannot satisfy the linear bounds in \eqref{E:cond_VLF}. 

Assuming that the polynomial Lyapunov functions, $V_i\,\forall\,i$, we found using the \textit{expanding interior algorithm} (Sec.\,\ref{subsec:decomp}) are quadratic, we can define $v_i:=\sqrt{V_i}\,,\,\forall\, i$,  which would satisfy the conditions in \eqref{E:cond_VLF} \cite{Weissenberger:1973,Jocic:1978}. In such a case, one needs to find the following bounds,
\begin{subequations}\label{E:cond_new}
\begin{align}
\forall i,\forall j\in\mathcal{N}_i\backslash\left\lbrace i\right\rbrace,\forall x_i&\in\mathcal{D}_i ,\forall x_j\in\mathcal{D}_j\,, \nonumber\\
{\eta}_{i1}\left\|x_i\right\|_2^2 &\leq V_i(x_i) \leq {\eta}_{i2}\left\|x_i\right\|_2^2 \,,\label{E:cond_new_1}\\
\nabla{V}_i^T\!\!f_i\, &\leq -{\eta}_{i3}\left\|x_i\right\|_2^2\,, \label{E:cond_new_2}\\
\text{and}~ \left\|\nabla{V}_i^T\!\!g_{ij}\right\|_2 &\leq {\zeta}_{ij}\left\|x_i\right\|_2\left\|x_j\right\|_2\,, \label{E:cond_new_3}
\end{align}
\end{subequations}
for some positive scalars ${\eta}_{i1},{\eta}_{i2},{\eta}_{i3},{\zeta}_{ij}$. Then using simple algebra the bounds in \eqref{E:cond_VLF} and \eqref{E:cond_inter} can be obtained as
\begin{subequations}\label{E:bounds}
\begin{align}
\forall i,\,\forall j\in\mathcal{N}_i\backslash\left\lbrace i\right\rbrace,~\tilde{\eta}_{i1} = \sqrt{{\eta}_{i1}} \,,~&\tilde{\eta}_{i2} = \sqrt{{\eta}_{i2}} \, ,	\\
			\tilde{\eta}_{i3} = \frac{{\eta}_{i3}}{2\sqrt{{\eta}_{i2}}} ~\text{and }&\tilde{\zeta}_{ij} = \frac{{\zeta_{ij}}}{2\sqrt{{\eta}_{i1}}}\,.		
\end{align}
\end{subequations}
Thus the computation of each element of the comparison matrix $\tilde{A}$ in \eqref{E:cond_A} requires multiple optimization steps.  

We may also note that some of the bounds in \eqref{E:cond_VLF} and \eqref{E:cond_inter}, while convenient for analytical insights, need not be optimal for computing a Hurwitz comparison matrix. For example, in \eqref{E:cond_inter}, $\left\|\nabla v_i^T g_{ij}\right\|_2$ is function of both $x_i$ and $x_j$ but is bounded by using only the norm on $x_j$.

\subsection{SOS Based Direct Approach}\label{subsec:stab_new}

We propose to use SOS methods to directly compute the comparison equation in \eqref{E:comparison}, in a decentralized way by calculating each row of $A=[a_{ij}]$ directly at each subsystem level. Note that, in \eqref{E:comparison}, we will be using quadratic (or, in general, polynomial) Lyapunov functions which do not satisfy the bounds \eqref{E:cond_VLF}-\eqref{E:cond_inter}. But we may observe that,
\begin{lemma}\label{L:equivalence}
If there exist Lyapunov functions $v_i\,$, for each $i\in\left\lbrace 1,2,\dots,m\right\rbrace,$ satisfying the comparison equation in \eqref{E:VAV}-\eqref{E:def_v}, for some matrix $\tilde{A}=[\tilde{a}_{ij}]$ with
\begin{align*}
\tilde{a}_{ij}\geq 0~\forall\,i\neq j\,,~\text{and}~\sum_{j=1}^m \tilde{a}_{ij}<0~\forall i\,,
\end{align*}
then there exists another matrix $A=[a_{ij}]$ satisfying the comparison equation \eqref{E:comparison_VAV} with $V_i:=v_i^2\,,\,\forall\,i\,$, with 
\begin{align*}
{a}_{ij}\geq 0~\forall\,i\neq j\,,~\text{and}~\sum_{j=1}^m {a}_{ij}<\sum_{j=1}^m \tilde{a}_{ij}<0~\forall i\,.
\end{align*}
\end{lemma}
\begin{proof}
Please refer to Appendix~\ref{A:proof}.
\end{proof}
It will be useful to note here that, an application of Gershgorin's Circle theorem \cite{Bell:1965} says that if a matrix with negative diagonal elements is strictly diagonally dominant\footnote{$A=[a_{ij}]$ is strictly diagonally dominant if $\sum_{j\neq i}\left|a_{ij}\right|<\left|a_{ii}\right|,\forall i$.} then the matrix is Hurwitz. 
We are now in a position to outline the SOS based procedure to directly compute the matrix $A=[a_{ij}]$ in the comparison equation \eqref{E:comparison}.

In this work, we are interested in $\mathcal{D}_i\,,\forall\,i$, of the form,
\begin{subequations}\label{E:Di}
\begin{align}
\forall i,~\mathcal{D}_i&:=\left\lbrace x_i\in\mathbb{R}^{n_i}\left| V_i(x_i)\leq \gamma_i^0\right.\right\rbrace, \quad \gamma_i^0\in\left(0,1\right), \\
\text{and,}~\mathcal{D}&:= \left\lbrace x\in\mathbb{R}^n\left| \bigcap_{i} V_i(x_i)\leq \gamma_i^0\right.\right\rbrace \, .\end{align}
\end{subequations}
Note that we exclude the boundary of the isolated subsystem ROA, $\gamma_i^0=1\,\forall\,i$, for reasons explained later. The comparison equation in \eqref{E:comparison_VAV} can then be translated into
\begin{subequations}\label{E:Vaij}
\begin{align}
\forall\,i\,,~\dot{V}_i(x_i)&\leq \sum_{j=1}^ma_{ij}V_j(x_j),~ \forall\,x \in \mathcal{D},\\
\text{i.e.,}~ \dot{V}_i&\leq  \sum_{j\in\mathcal{N}_i}a_{ij}V_j,~ \text{when}~V_j\leq \gamma_j^0~\forall j\in\mathcal{N}_i\,,
\end{align}\end{subequations}
since we know that $a_{ij}=0~\forall j\notin\mathcal{N}_i$. Using the Positivstellensatz theorem (Theorem~\ref{T:Putinar}), with $u_i:=\left(\gamma_i^0-V_i(x_i)\right)$ and $\mathcal{K}=\mathcal{D}$, we can cast \eqref{E:Vaij} into an SOS feasibility problem,
\begin{subequations}\label{E:Vaij_SOS}
\begin{align}
&- \nabla V_i^T\left(f_i+g_i\right) + \sum_{j\in\mathcal{N}_i} \left(a_{ij}V_j - \sigma_{ij}\left(\gamma_j^0-V_j \right)\right) \in \Sigma[\bar{x}_i] \\
&~\text{with}\quad\sigma_{ij}\in\Sigma[\bar{x}_i],~\forall i\in\left\lbrace 1,2,\dots,.m\right\rbrace, \forall j\in\mathcal{N}_i\,.
\end{align}\end{subequations}
where $\bar{x}_i$ was defined in \eqref{E:Ni}. The goal is to find the `optimal' scalars $a_{ij}\,\forall i,j\in\mathcal{N}_i$ satisfying \eqref{E:Vaij_SOS} so as to obtain the tightest possible bound in \eqref{E:comparison_VAV}. We can thus formulate the following SOS optimization problem,
\begin{align}\label{E:sos_A}
\forall i\in\left\lbrace 1,2,\dots,m\right\rbrace,\quad \min_{\sigma_{ij}}  \sum_{j\in\mathcal{N}_i}a_{ij}\,,~\text{subject to} ~\eqref{E:Vaij_SOS}\,.
\end{align}
This simple SOS formulation helps us find the comparison equation \eqref{E:comparison} in a decentralized way, by computing each row of the comparison matrix $A=[a_{ij}]$ in a single optimization problem at each subsystem level. The optimization problem can be easily implemented on a parallel platform, with the complexity of the problem essentially dependent on the size of the largest neighborhood $\mathcal{N}_i$.

Further note that, if the minimal values of all the row-sums in \eqref{E:sos_A} are negative, then the matrix $A=[a_{ij}]$ thus found is a strictly diagonally dominant with negative diagonal entries, and hence, Hurwitz \cite{Bell:1965}. However, if $\sum_{j=1}^ma_{ij}\geq0$ for any $i$, then the eigenvalues of $A=[a_{ij}]$ need to be computed. 

Finally a note on invariance of the domain $\mathcal{D}$, which along with the presence of Hurwitz $A=[a_{ij}]$ guarantees that $\mathcal{D}$ is a domain of exponential stability, as noted in \eqref{E:comparison}. According to \cite{Weissenberger:1973}, an estimate of the ROA can be given by
\begin{subequations}
\begin{align}
\mathcal{R} := &\left\lbrace x\in\mathcal{D}\left| \max_i\left(\frac{V_i(x_i)}{p_i}\right)\leq \min_j\left(\frac{\gamma_j^0}{p_j}\right)\right.\right\rbrace\,, \\
\text{where,}~&p_i>0,\quad\forall i \in\left\lbrace 1,2,\dots,m\right\rbrace\,,\\
\text{and}~	& Ap<0,\quad p := \left(p_1,p_2,\dots,p_m\right)^T\,.
\end{align}
\end{subequations}
Then it is easy to see that,
\begin{align}\label{E:RAD}
\mathcal{R} \equiv \mathcal{D},~\text{if}~A\gamma^0<0,~\gamma^0:=\left(\gamma_1^0,\gamma_2^0,\dots,\gamma_m^0\right)^T.
\end{align}

\section{NUMERICAL EXAMPLE}\label{S:example}

 \subsection{Model Description}\label{S:model}
 {W}{e} consider a network of nine Van der Pol `oscillators' \cite{van1926}, with parameters of each oscillator chosen to make them individually stable. Each Van der Pol oscillator constitutes a subsystem, with the interconnections shown below
\begin{align}
\begin{array}{lll}
\mathcal{N}_1:\left\lbrace 1, 2, 5, 9\right\rbrace & \mathcal{N}_2:\left\lbrace 2, 1, 3\right\rbrace & \mathcal{N}_3:\left\lbrace 3, 2, 8\right\rbrace\\ 
\mathcal{N}_4:\left\lbrace 4, 6, 7\right\rbrace & \mathcal{N}_5:\left\lbrace 5, 1, 6\right\rbrace & \mathcal{N}_6:\left\lbrace 6, 4, 5\right\rbrace\\
\mathcal{N}_7:\left\lbrace 7, 4, 8, 9\right\rbrace & \mathcal{N}_8:\left\lbrace 8, 3, 7\right\rbrace & \mathcal{N}_9:\left\lbrace 9, 1, 7\right\rbrace\,.
\end{array}
\end{align}
The dynamics of each oscillator, in presence of the neighbor interactions, is given by  
 \begin{subequations}
\begin{align}
\forall j\in&\left\lbrace 1,2,\dots,9\right\rbrace,\notag\\
	&\dot{x}_{j,1}= x_{j,2}\, \\
	&\dot{x}_{j,2}=\mu_jx_{j,2}\left(1-x_{j,1}^2\right) - x_{j,1} +x_{j,1}\!\!\!\!\!\!\sum_{k\in\mathcal{N}_j\backslash\left\lbrace j\right\rbrace}\!\!\!\!\!\!\beta_{jk}x_{k,2}\,,
\end{align}
\end{subequations}
where $\mu_j\,,\,\forall j$, are chosen randomly from $\left(-3\,,\,-1\right)$ and the interaction coefficients $\beta_{jk}\,,\,\forall j,\forall k\!\!\in\!\!\mathcal{N}_j\backslash\left\lbrace j\right\rbrace$, are chosen randomly from $\left(-0.4\,,\,0.4\right)$. 

Using the expanding interior algorithm, we find estimates of the ROAs of the isolated, or `free', subsystems via quadratic Lyapunov functions. As an example, Fig.\,\ref{F:ROAcompare} shows a comparison of the true ROA of the isolated subsystem 9 and an estimate using a quadratic Lyapunov function,
\begin{subequations}\begin{align}
\mathcal{R}_9^0&=\left\lbrace \left(x_{9,1},x_{9,2}\right)\left\vert~ V_9\leq 1\right.\right\rbrace, \\
\text{where,}~V_9&=0.595\,x_{9,1}^2 + 0.227\,x_{9,1}\,x_{9,2} + 0.520\,x_{9,2}^2\,.
\end{align}\end{subequations}
\begin{figure}[thpb]
      \centering
	\includegraphics[scale=0.5]{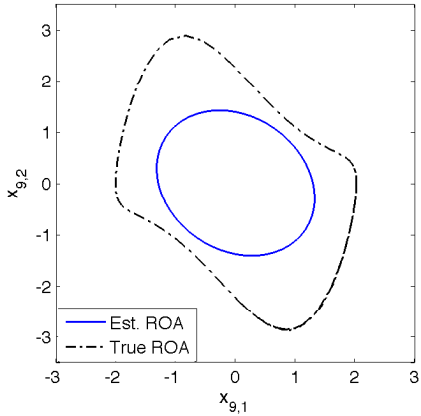}
      \caption{Comparison of estimated and true ROA for isolated subsystem 9.}
      \label{F:ROAcompare}
   \end{figure}   
   
\subsection{Exponential Stability of Isolated Subsystems}
Existence of a comparison matrix $A=[a_{ij}]$ requires that the diagonal entries $a_{ii}$ are negative, which necessitates that 
\begin{align}\label{E:alpha}
\forall i\,,~\exists\, \alpha_i>0\,,~\text{so that}~\nabla V_i^Tf_i\leq -\alpha_i\,V_i\,,~\forall\,x_i\in\mathcal{D}_i
\end{align}
where $\mathcal{D}_i\,,\,\forall i$, were defined in \eqref{E:Di}. Note that the condition \eqref{E:alpha} is a sufficient condition of exponential stability for the isolated subsystems, as in \eqref{E:Lyap_alpha}. We can use SOS optimization, similar to \eqref{E:Vaij_SOS}-\eqref{E:sos_A}, to find the maximal $\alpha_i\,,\,\forall i$, the `self-decay rates', for a set of given $\gamma_i^0\,,\,\forall i$. Higher values of $\alpha_i$ indicates better chance of finding a Hurwitz comparison matrix. In Fig.\,\ref{F:decay} we show the variations of $\alpha_i$ for each $i$, when the initial level set $\gamma_i^0$ is varied from 0 to 1. For each subsystem, as $\gamma_i^0$ approaches 1, $\alpha_i$ approaches 0. This shows that it is not possible to obtain a Hurwitz comparison matrix when the initial conditions lie close to the boundary of the estimated ROAs, and hence the exclusion of $\gamma_i^0=1$ in \eqref{E:Di}.

\begin{figure}[thpb]
      \centering
	\includegraphics[scale=0.5]{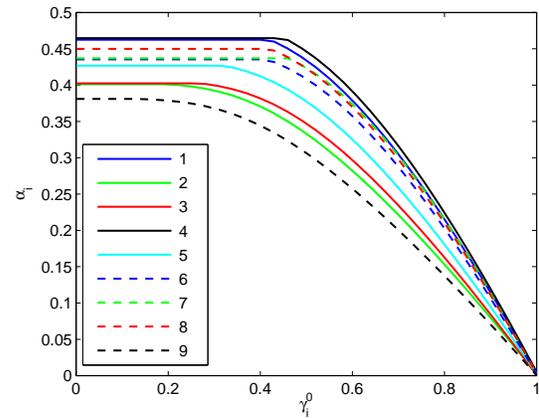}
      \caption{Evolution of self-decay rates against varying initial level-sets.}
      \label{F:decay}
   \end{figure}

\subsection{Comparison Equation}
We recall that two sufficient criteria for a domain $\mathcal{D}$ to be an estimate of the ROA, are that the comparison matrix in \eqref{E:comparison} is Hurwitz and $\mathcal{D}$ is an invariant domain under the dynamics \eqref{E:fi}. To compare the performance of the traditional approach and the direct approach, we need to monitor how well the above mentioned criteria are satisfied for a set of arbitrarily chosen $\mathcal{D}$. 

While this would require an exhaustive simulation over all possible domains $\mathcal{D}$ defined in \eqref{E:Di}, we choose to examine only those $\mathcal{D}$ where $\gamma_1^0=\gamma_2^0=\dots=\gamma_9^0=\gamma^*$, i.e. 
\begin{align}\label{E:Deq}
\mathcal{D}:=\left\lbrace x\in\mathbb{R}^9\left\vert \bigcap_{i=1}^9 V_1\leq \gamma^*\right.\right\rbrace
\end{align}
for some $\gamma^*\!\!\in\!\!\left(0,1\right)$. For each $\gamma^*$, and domain $\mathcal{D}$, we compute the comparison matrices using the traditional and the direct approach. Denoting by $\text{Re}\left(\lambda\right)$ the real parts of the eigenvalues of a matrix, we note that if the maximum of $\text{Re}(\lambda)$ is negative, then the matrix is Hurwitz. Further, by applying \eqref{E:RAD}, the domain \eqref{E:Deq} is guaranteed to be invariant if the maximum row-sum of the comparison matrix is negative. Fig.\,\ref{F:CompMat} shows an evolution of these two properties (maximum $\text{Re}(\lambda)$ and maximum row-sum) for the comparison matrices, computed using the two approaches, for a range of $\gamma^*$.
\begin{figure}[thpb]
      \centering
	\includegraphics[scale=0.5]{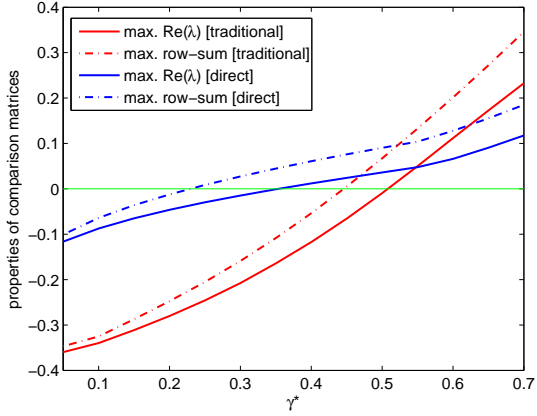}
      \caption{Evolution of the properties of comparison matrices, computed using traditional and direct approach, against varying initial level-sets.}
      \label{F:CompMat}
   \end{figure}

We note that both the maximum row-sum and the maximum $\text{Re}(\lambda)$ generally increases as $\gamma^*$ increases from 0 to 1, indicating that as the domain $\mathcal{D}$ `expands', it becomes more difficult to certify stability. We also note that, for both approaches, the maximum row-sum becomes positive before maximum $\text{Re}(\lambda)$, indicating that the `invariance' criterion is lost before the `Hurwitz' criterion. Significantly, we also note that both the Hurwitz and invariance criteria are satisfied for a wider range of $\gamma^*$ in case of the direct approach than in the case of the traditional one. Thus, with regards to both the criteria, the direct approach is seen to perform better than the traditional approach. 

\subsection{Test Case}
Let us illustrate how this method can be used to certify exponential convergence of a given initial condition to the origin. The system dynamics is evolved against a randomly generated initial condition, and is found to be converging to the origin. Fig.\,\ref{F:states} shows the evolution of the states belonging to subsystems - 2,\,6,\,7 and 8. 
\begin{figure}[thpb]
      \centering
	\includegraphics[scale=0.5]{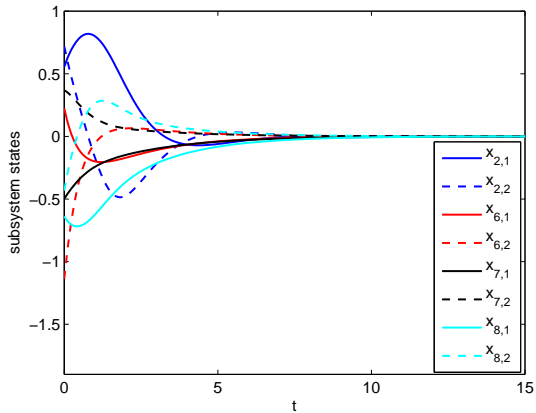}
      \caption{Evolution of subsystem states under an arbitrary disturbance.}
      \label{F:states}
   \end{figure}
   The initial condition yields the following level sets,
\begin{align*}
\gamma^0 &= \left[ \,0.08\,,\,    0.56\,,\,    0.58\,,\,    0.31\,,\,    0.08\,,\,    0.61\,,\,    0.18\,,\,    0.45\,,\,    0.14\,\right]^T
\end{align*} 
which is then used to define the domain $\mathcal{D}$, in \eqref{E:Di}. Then the SOS-based direct approach is used to compute the comparison matrix, $A\in\mathbb{R}^{9\times 9}$, with maximum $\text{Re}(\lambda)$ as -0.078, and $A\gamma^0<0$. The solution, $w(t)\in\mathbb{R}^9$, of the corresponding comparison equation $\dot{w}=A\,w\,,\,w(0)=\gamma^0$, is plotted against the actual Lyapunov level sets in Fig.\,\ref{F:level}, for subsystems 2, 6, 7 and 8.
\begin{figure}[thpb]
      \centering
	\includegraphics[scale=0.5]{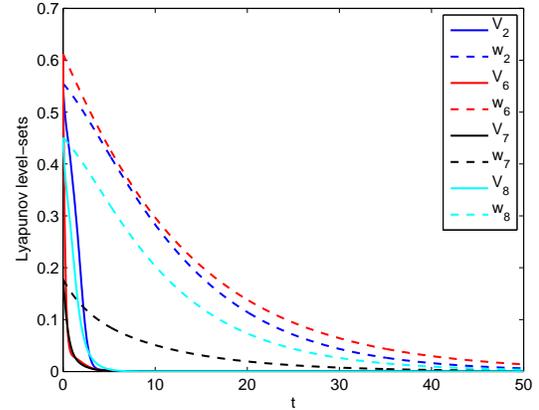}
      \caption{Comparison of the actual Lyapunov level sets and their upper bounds from the linear comparison equation.}
      \label{F:level}
   \end{figure}
The trajectories $w(t)$ exponentially converge to zero and, from Lemma\,\ref{L:comparison}, provide an upper bound on the corresponding subsystem Lyapunov function level sets.

When the same procedure is done with the traditional approach, we obtain a Hurwitz comparison matrix, $\tilde{A}$, with maximum $\text{Re}(\lambda)$ as -0.001, but with $\tilde{A}\left(\gamma^0\right)^{1/2}>0$, thus violating the invariance condition.

\section{CONCLUSIONS AND FUTURE WORKS}\label{S:conclusion}
\subsection{Conclusions}
We have presented an SOS based direct approach to compute the linear comparison principle for stability analysis of interconnected systems. We have also discussed the traditional approach to obtaining the comparison equations, and shown how the direct approach can yield `better', or less conservative, certificates of exponential stability. Using a network of Van der Pol systems we have presented a comparison of the two approaches. The proposed approach can be implemented on a suitable parallel platform where each row of the comparison matrix, corresponding to each subsystem, is computed in parallel.

\subsection{Future Works}
A decentralized control framework can be visualized where each subsystem computes a local control law that will guarantee satisfaction of the Hurwitz and invariance conditions. SOS methods can be used to extend the stability analysis to higher order, and more general, comparison equations. Also, it would be interesting to see how the use of higher order (for example, quartic) Lyapunov functions in the comparison equation affects the conservativeness of the stability certificates.

\bibliographystyle{IEEEtran}
\bibliography{references}

\appendix

\subsection{Proof of Lemma~\ref{L:equivalence}}\label{A:proof}
Since $V_i=v_i^2~\forall\,i$ and $\tilde{a}_{ij}>0~\forall\,i\neq j$, we have
\begin{align}
\forall i\in\left\lbrace 1,2,\dots,m\right\rbrace\,,~\dot{V}_i &\leq 2v_i\sum_{j=1}^ma_{ij}v_j\notag\\
		& \leq 2\tilde{a}_{ii}V_i + 2\sum_{j\neq i}\tilde{a}_{ij}v_iv_j \notag\\
		& \leq 2\tilde{a}_{ii}V_i + \sum_{j\neq i}\tilde{a}_{ij}\left(V_i+V_j\right)\notag\\
		& = \left(\tilde{a}_{ii}+\sum_{j=1}^m\tilde{a}_{ij}\right)V_i + \sum_{j\neq i}\tilde{a}_{ij}V_j
\end{align}
Choosing $a_{ii}=\left(\tilde{a}_{ii}+\sum_{j=1}^m\tilde{a}_{ij}\right)~\forall\,i$, and $a_{ij}=\tilde{a}_{ij}~\forall i\neq j$, and recalling that $\sum_{j=1}^m\tilde{a}_{ij}<0$ we may conclude the proof.

\addtolength{\textheight}{-3cm}   

\end{document}